\documentclass[a4paper,12pt]{amsart}

\usepackage{amsthm,amssymb,amsmath,amscd}
\usepackage{amsthm}
\usepackage{pgf, tikz}
\usepackage{graphicx,subcaption}
\usepackage{verbatim}
\theoremstyle{definition}
\newtheorem{definition}{Definition}%[section]

\theoremstyle{plain}
\newtheorem{theorem}[definition]{Theorem}
{Lemma}
{Theorem}
\newtheorem{lemma}[definition]{Lemma}

\newtheorem{prop}[definition]{Proposition}

\newtheorem{fact}[definition]{Fact}
\newtheorem{claim}[definition]{Claim}
{Conjecture}

\theoremstyle{remark}

\newcommand{\pp}{P}
\newcommand{\ppk}{P^{(k)}}

\begin{document}

\title[Ramsey number for $k$-paths]
{On multicolor Ramsey numbers\\ for loose $k$-paths of length three}

\author[T.~{\L}uczak]{Tomasz {\L}uczak}

\address{Adam Mickiewicz University,
Faculty of Mathematics and Computer Science
ul.~Umultowska 87,
61-614 Pozna\'n, Poland}

\email{\tt tomasz@amu.edu.pl}

\author[J.~Polcyn]{Joanna Polcyn}

\address{Adam Mickiewicz University,
Faculty of Mathematics and Computer Science
ul.~Umultowska 87,
61-614 Pozna\'n, Poland}

\email{\tt joaska@amu.edu.pl}

\author[A.~Ruci\'nski]{Andrzej Ruci\'nski}

\address{Adam Mickiewicz University,
Faculty of Mathematics and Computer Science
ul.~Umultowska 87,
61-614 Pozna\'n, Poland}

\email{\tt rucinski@amu.edu.pl}

\thanks{The first author
supported by Polish NSC grant 2012/06/A/ST1/00261. The third author supported by the Polish NSC grant 2014/15/B/ST1/01688}



\date{March 18, 2017}

\begin{abstract}
We show that there exists an absolute constant $A$ such that for each $k\ge2$ and every coloring of the edges of the complete $k$-uniform hypergraph on $ Ar$ vertices
with $r$ colors, one of the color classes contains a loose path of length three.
\end{abstract}

\maketitle

\section{Introduction} For $k\ge2$, a \emph{$k$-uniform} hypergraph (or, briefly, a $k$-graph) is an ordered pair $H=(V,E)$, where $V=V(H)$ is a finite set and $E=E(H)$ is a subset of the set $\binom Vk$ of all $k$-element subsets of $V$. If $E=\binom Vk$,  we call $H$ \emph{complete} and denote it by $K^{(k)}_n$, where $n=|V|$. The elements of $V$ and $E$ are called, respectively, the vertices and edges of $H$. We often identify $H$ with $E(H)$, writing, for instance, $|H|$ instead of $|E(H)|$. The degree $\deg_H(v)$ of a vertex $v$ in $H$ equals the number of edges of $H$ which contain $v$.  \emph{A star} is a $k$-graph $S$ with a vertex $v$ contained in all the edges of $S$. A star is full if it consists of all sets in $\binom Vk$ containing~$v$, that is, in which $\deg_S(v)=\binom{n-1}{k-1}$.

Let $\ppk$ denote a loose $k$-uniform path (shortly, a $k$-path) of length three, that is, the only connected $k$-graph on $3k-2$ vertices with three edges and no vertex of degree three. In this paper we study the multicolor Ramsey number $R(\ppk;r)$ for
$\ppk$ defined as  the smallest number $n$ such that each
coloring of the edges of the complete $k$-graph $K^{(k)}_n$ with $r$ colors
leads to a monochromatic copy of $\ppk$. In the graph case, i.e. when $k=2$, it
is easy to see that  the value of $R(\pp;r)$
 is equal to $2r+c_r$, where $c_r\in\{0,1,2\}$ and depends on the divisibility  of $r$ by three (see \cite{Radz}). On the other hand,  coloring the edges of $K_{r+3k-4}^{(k)}$ by $r-1$ stars and one clique $K_{3k-3}^{(k)}$ (e.g., \cite{J}, Proposition~3.1) shows that
\begin{equation}\label{eq1}
R(\ppk;r)\ge  r+3k-3.
 \end{equation}
It is conjectured that for each $k\ge 3$  and all $r$ there is equality in~$(\ref{eq1})$. So far,  it has been verified only for $k=3$ and $r=2,3,\dots, 10$ (\cite{GR, J, JPR, PR, P}). In fact, for $k=3$ and $r=2$ the Ramsey number has been determined for paths of all lengths \cite{Omidi}.

A general upper bound on $ R(\ppk;r)$, $k\ge3$, follows by a standard application of Tur\'an numbers. Indeed, it was proved by F\"uredi, Jiang, and Seiver~\cite{FJS} that for $n\ge n_0(k)$
the unique largest
$\ppk$-free $k$-graph on $n$ vertices is a full star which contains at most $\binom {n-1}{k-1}$ edges. From this it follows that for $r$ large enough $R(\ppk;r)\le kr +1$ and if
we use the fact that the extremal graph is unique we get (see \cite{J}, Proposition~3.2)
\begin{equation}\label{eq2}
R(\ppk;r)\le  kr,
 \end{equation}
 valid for all $k\ge3$ and $r\ge r_0(k)$. Similar results for loose cycles of length three were obtained by Gy\'{a}rf\'{a}s and Raeisi~\cite{GR}.
For $k=3$,
it was proved by {\L}uczak and Polcyn~\cite{LP} that $R(P^{(3)};r)\le 2r+O(\sqrt r)$. The main goal of this paper is to show that for $r$ large enough $R(\ppk;r)/r$ is bounded from above by a constant which does not depend on $k$.

\begin{theorem}\label{thm1}
For each $k\ge 3$ there exists $r_k$ such that for all $r\ge r_k$
\begin{equation*}%\label{eq3}
R(\ppk;r)\le 250r.
 \end{equation*}
\end{theorem}

\section{Proof of Theorem~\ref{thm1}}

In view of (\ref{eq2}), we may restrict ourselves to $k\ge250$. Our proof uses two results on Tur\'an numbers for loose $k$-paths of length two and three. The first of them was proved by Keevash, Mubayi, and Wilson~\cite{KMW}.

\begin{lemma}\label{nsi}
	Let  $k\ge 4$ and  $H$ be a k-uniform hypergraph on $n$ vertices in which no two edges intersect on a single vertex.
 Then, for large $n$, $|H|\le {n-2\choose k-2}$.
\end{lemma}
%Lemma \ref{nsi} is not true for $k=3$. It is folklore, however, that the largest number of edges in an $n$-vertex triple system without a pair of %edges intersecting on a single vertex is $n$.

The second result, due to F\"uredi, Jiang, and Seiver~\cite{FJS},
 %, mentioned already in the Introduction,
deals with the main object of our study: $\ppk$, the loose $k$-uniform path of length three.

\begin{lemma}\label{pk} Let $k\ge 3$ and  $H$ be a $\ppk$-free $k$-uniform hypergraph on $n$ vertices.
Then, for large $n$,  $|H|\le\binom {n-1}{k-1}$.
\end{lemma}
\noindent It is also proved in \cite{FJS} that, in fact, the only extremal $k$-graph is a full star.
Theorem \ref{thm1} is a direct consequence  of the following `stability' version of  Lemma \ref{pk} which states, roughly, that the structure of each
 $\ppk$-free dense $k$-graph is dominated by a giant star.

\begin{lemma}\label{l1}
For every $k\ge250$ and $n\ge n_0(k)$, each $\ppk$-free $k$-uniform hypergraph $H$,
has a vertex $v$ with degree
$$\deg_H(v)\ge |H|-{0.96}^k\binom {n-1}{k-1}.$$
\end{lemma}

We defer the proof of Lemma~\ref{l1} to the next section. Here we show how Theorem~\ref{thm1} follows from it.

\begin{proof}[Proof of Theorem~\ref{thm1}] For a given $k\ge 250$ and $A=250$
 let $r\ge r_k$, where $r_k$ is chosen so  that $250r_k\ge n_0(k)$ with
$n_0(k)$ defined as in  Lemma~\ref{l1}.
Suppose that the complete $k$-graph $K:=K_{Ar}^{(k)}$ on $Ar$ vertices  is colored with  colors $1,2,\dots,r$ in such a way that  no monochromatic $\ppk$ emerges.  For every color $c$  choose (possibly with repetitions) a vertex $v_c$ with maximum degree in this color and let $R=\{v_c:\; c=1,2,\dots,r\}$.

Consider now the complete $k$-graph $H$ obtained from $K$ by
 removing all vertices in $R$. We have
$|V(H)|\ge Ar-|R|\ge (A-1)r$ and thus  $|E(H)|\ge\binom {(A-1)r}{k}$. On the other hand, by applying Lemma~\ref{l1} to each color class,  we have
$|H|\le r(0.96)^k\binom {Ar-1}{k-1}$. On the other hand, since $k\ge A= 250$,  we have
\begin{equation}\label{A<A} r(0.96)^k\binom {Ar-1}{k-1}< \binom {(A-1)r}{k}\,,\end{equation}
a contradiction.
To see (\ref{A<A}), observe first that the two sides of (\ref{A<A}) are asymptotic (as $r$ is growing) to, respectively, $0.96^kA^{k-1} r^k/(k-1)!$ and $(A-1)^kr^k/k!$.
Thus it remains to show that $(A-1)^k>k(0.96)^kA^{k-1}$, or, equivalently, $(A-1)(1-1/A)^{k-1}>k(0.96)^k$. Now it is enough to observe that  $1-1/A\ge 0.99$ for $A\ge 250$
and $k<(99/96)^k$ for $k\ge167$.
\end{proof}

\section{Proof of Lemma~\ref{l1}}

We begin by stating two elementary facts the short proofs of which are provided for completeness.

\begin{fact}\label{f1}
Every hypergraph $H$  contains a sub-hypergraph $G$ with minimum degree greater than $\frac{|E(H)|}{|V(H)|}$.
\end{fact}

\begin{proof}
	Define $G$ as a subhypergraph of $H$ which maximizes the ratio $\frac{|E(G)|}{|V(G)|}$ and has the smallest number of vertices.
If for some $v\in V(G)$,  $\deg_G(v)\le\frac{|E(H)|}{|V(H)|}$, then
$$
\frac{|E(G-v)|}{|V(G-v)|}\ge\frac{|E(G)|-|E(H)|/|V(H)|}{|V(G)|-1}\ge \frac{|E(G)|}{|V(G)|},
$$
which contradicts our choice of $G$.
\end{proof}

\begin{fact}\label{f2}
Every bipartite graph $B$ with vertex classes $V_1$ and $V_2$  contains a subgraph $G$ with $\deg_G(v)\ge |B|/(2|V_i|)$ for every vertex $v\in V(G)\cap V_i$, $i=1,2$.
\end{fact}
\begin{proof}
Let us  remove one by one the vertices with (current) degree smaller than the above bounds.
	Then, by the time the degrees of all remaining vertices satisfy the required bounds,
	we remove fewer than
$$|V_1|\times |B|/(2|V_1|)+|V_2|\times |B|/(2|V_2|)=|B|$$ edges, and so the final subgraph $G$ is non-empty.
\end{proof}

%\begin{fact}\label{f11}
%For $\ell \ge 2$, if an $\ell$-hypergraph $H$ with $m$ vertices has at least $b\binom n\ell$ edges, then $m\ge b^{1/\ell}n$.
%\end{fact}
%\begin{proof} Trivially, we must have $\binom m\ell\ge b\binom n\ell$, which implies the desired bound on $m$.
%\end{proof}

Lemma~\ref{l1} is a straightforward consequence of the following two propositions.

\begin{prop}\label{f2}
For all $k\ge3$, $b>0$, and sufficiently large $n$, the following holds. Let $H$ be a $\ppk$-free, $n$-vertex, $k$-uniform hypergraph and let, for some $v\in V(H)$, $\deg_H(v)\ge b\binom{n-1}{k-1}$. Then,
$$\deg_H(v)\ge|H|-\left(1-\left(\frac b{k-1}\right)^{1/(k-2)}\right)^{k-1}\binom {n-1}{k-1}.$$
\end{prop}

\begin{proof} Let $H(v)$ be the link of $v$ in $H$, that is, the $(k-1)$-uniform, $(n-1)$-vertex hypergraph consisting of all $(k-1)$-element subsets of $V(H)$ which together with $v$ form edges in $H$. Note that $|H(v)|=\deg_H(v)$. Fact \ref{f1} implies that there is a subgraph $F$ of $H(v)$ with minimum degree
$$\delta(F)\ge\delta:=\frac b{k-1}\binom {n-2}{k-2}.$$

\begin{claim}\label{mm}
The number of vertices $|V(F)|$ of $F$ is bounded from below by
\begin{equation*}\label{m}
|V(F)|\ge \left(\frac b{k-1}\right)^{1/(k-2)}(n-1).
\end{equation*}
\end{claim}

\begin{proof}
Since
$$\binom{|V(F)|}{k-1}\ge |F|\ge |V(F)|\frac b{(k-1)^2}\binom{n-2}{k-2},$$
it follows that
$$\binom{|V(F)|-1}{k-2}\ge  \frac b{k-1}\binom{n-2}{k-2},$$
so,
$$1\ge \frac b{k-1}\frac{(n-2)_{k-2}}{(|V(F)|-1)_{k-2}}>\frac b{k-1}\left(\frac{n-1}{|V(F)|}\right)^{k-2},$$
which implies the required bound for $|V(F)|$.
\end{proof}

\begin{claim}\label{disjoint}
Let $n$ be sufficiently large. For every edge $e\in H$, either $v\in e$ or $e\cap V(F)=\emptyset$.
\end{claim}

\proof Suppose there exists an edge $e\in H$ such that $v\not\in e$ and $e\cap V(F)\neq\emptyset$. Let $w\in e\cap V(F)$. Since $\deg_F(w)\ge \delta=\Omega(n^{k-2})$ while the number of edges of $F$ intersecting $e$ on at least two vertices is $O(n^{k-3})$, there is an edge $f'\in F$ such that $e\cap f'=\{w\}$. Further, since $\deg_H(v)\ge b\binom{n-1}{k-1}$ while the number of edges of $H$ containing $v$ and intersecting $e\cup f'$ is $O(n^{k-2})$, there is an edge $h\in H$ such that $v\in h$ and $h\cap(e\cup f')=\emptyset$. The edges $e$, $f'\cup\{v\},$ and $h$ form a copy of $\ppk$ in $H$, a contradiction. \qed

\bigskip

In view of Claim \ref{disjoint}, to complete the proof of Proposition \ref{f2}, we bound from above
the number of edges of $H$ which do not contain $v$  by $|H-V(F)|$, where $H-V(F)$ is the induced subhypergraph of $H$ obtained by deleting all vertices of $F$. Since $H$, and thus $H-V(F)$, is $\ppk$-free, we can bound $|H-V(F)|$ by the Tur\'an number for $\ppk$ given in Lemma \ref{pk}.
Using the bound for $|V(F)|$ given by Claim~\ref{mm}, we get
\begin{align*}
|H-V(F)|&\le\binom {n-|V(F)|-2}{k-1}\\
&\le \binom{n-(n-1)(b/(k-1))^{1/(k-2)}-2}{k-1}\\
&<\binom{(n-1)(1-(b/(k-1))^{1/(k-2)})}{k-1}\\
&\le \left(1-\left(\frac b{k-1}\right)^{1/(k-2)}\right)^{k-1}\binom {n-1}{k-1}\,.\qed
\end{align*}
\renewcommand{\qed}{}
\end{proof}

\begin{prop}\label{relaxed} For all $k\ge250$  and sufficiently large $n$ the following holds.
If $H$ is a $\ppk$-free $k$-graph on $n$ vertices and $|H|\ge 0.96^k \binom {n-1}{k-1}$, then $\Delta(H)\ge0.9^k \binom {n-1}{k-1}$.
\end{prop}

\begin{proof}
  Let $H$ be a $\ppk$-free $k$-graph on $n$ vertices and with $|H|\ge 0.96^k \binom {n-1}{k-1}$.
By $F$ we denote the shadow of $H$, i.e.
 $$F=\{ f\in [n]^{k-1}: f\subset e \textrm{\ for some\ }e\in H\}\,. $$

Let us now suppose that $\Delta(H)<0.9^k \binom {n-1}{k-1}$. We shall show that this assumption
leads to a contradiction.

The main idea of the argument goes roughly as follows.
First  we deal with the case when $F$ is small (Claim~\ref{cl3} below). Then there are many
vertices $v$ with large  links. Consequently, it is enough to find in $F$ a loose
$(k-1)$-path of
length two, say $f_1$, $f_2$ (and for that, due to Lemma~\ref{nsi} we only require that $|F|=\Omega(n^{k-3})$) and find another $f_3$ in $F$ so that $(f_1\cup f_2)\cap f_3=\emptyset$.
Then, for some $v_1,v_2\in V(H)$, the edges $\{v_1\}\cup f_1$,
$\{v_2\}\cup f_2$, and $\{v_2\}\cup f_3$ form a $\ppk$ in $H$.
In the case  $F$ is large we select the   three disjoint subsets of vertices, $W_1$, $W_2$ and $W_3$,
such that the unions $S_i$
of the links   of vertices in $W_i$ are edge-disjoint  and roughly of the same size, for each $i=1,2,3$.
Since $F$ is large, so are all $S_i$'s; in fact each of them
covers a majority of the vertices of $H$. Thus, one can find a $(k-1)$-path of length three
consisting of some sets $f_1$, $f_2$, $f_3$, where $f_i\in S_i$ for $i=1,2,3$, which in turn,
can be easily extended  to~$\ppk$.

In order to make the above precise, let us start with the following observation.

  \begin{claim}\label{cl3}
$|F|\ge \frac14|H|$.
\end{claim}
\proof
  Let us consider an auxiliary bipartite graph $B$, with vertex classes $V(H)$ and $F$, and with edge set
  $$\{\{v,f\}:\,\{v\}\cup f\in H\}\,.$$
 Clearly, $|B|=k|H|$.
 Further, define
 $$F'=\{f\in F: |\{e\in H: f\subset e\}|\ge 2k\}\,$$
 and  observe that $|F'|\le \binom{n-1}{k-2}$. Indeed, otherwise, by the Tur\'an number for $P^{(k-1)}$, $F'$ would contain a copy of $P^{(k-1)}$ which could be
easily  extended to a copy of $\ppk$ in $H$.

Let $B'$ be the subgraph of $B$ consisting of all edges with one endpoint in $F'$. We have
$$|B|=\sum_f\deg_B(f)\le |B'|+(|F|-|F'|)2k,$$
so
$$|F|\ge|F|-|F'|\ge\frac1{2k}(|B|-|B'|).$$
Thus, recalling that $|B|=k|H|$, to complete the proof of Claim \ref{cl3}, it suffices to show that
\begin{equation}\label{EE}
|B'|\le |B|/2.
\end{equation}
Suppose that $|B'|\ge |B|/2$. Then,
$$|B'|\ge \frac k2|H|\ge \frac k2(0.96)^k\binom{n-1}{k-1}.$$
We apply Fact \ref{f2} to $B'$, obtaining a subgraph $B''$ with vertex sets $V_1\subset V(H)$ and $F''\subset F'$
such that each vertex $v\in V_1$ has in $B''$ degree at least $\tfrac14(0.96)^k \binom {n-1}{k-2}$ and each $f\in F''$ has in $B''$ degree at least $\tfrac14(0.96)^k n$.

Since, for large $n$, $|F''|\ge\tfrac14(0.96)^k \binom {n-1}{k-2} > {n-2\choose k-3}$, by Lemma \ref{nsi}, $F''$ contains two $(k-1)$-sets $f_1$, $f_2$
such that $|f_1\cap f_2|=1$.
Let $N_i$ be the neighborhood of $f_i$ in $B''$, $i=1,2$. If there was an edge $(v,f)\in B''$ with  $v\in N_1\cup N_2$ (say, $v\in N_2$) and $(f_1\cup f_2)\cap f=\emptyset$, then the $k$-sets $\{v_1\}\cup f_1$, $\{v\}\cup f_2$, and $\{v\}\cup f$, where $v_1\in N_1$, $v_1\neq v$, would form a copy of $\ppk$ in $H$, a contradiction. Thus, in $B''$, all neighbors $f$ of vertices  in $N_1\cup N_2$ must intersect $f_1\cup f_2$. Since $|N_1\cup N_2|\ge |N_1|\ge \tfrac14(0.96)^k n$, the number of edges of $B''$ leaving $N_1\cup N_2$ is at least
$$\frac14(0.96)^k n\times \frac14 0.96^k \binom {n-1}{k-2}=\frac1{16}(0.96)^{2k}n\binom {n-1}{k-2}.$$
Each of these edges of $B''$ represents an edge of $H$ which intersects $f_1\cup f_2$, a set of size smaller than $2k$. Hence, by averaging, there exists a vertex in $f_1\cup f_2$  belonging to at least
$$\frac1{32k}(0.96)^{2k}n\binom {n-1}{k-2}>0.9^k\binom{n-1}{k-1}$$
of these edges (note that the last inequality is valid for  $k\ge250$). This contradicts our assumption on  $\Delta(H)$ and, therefore, completes the proof of Claim \ref{cl3}. \qed

\bigskip

To continue with the proof of Proposition \ref{relaxed}, for every $f\in F$
we choose just one vertex $v_f$ such that $\{v_f\}\cup f\in H$. Observe that by our assumption on $\Delta(H)$, for each $v\in V(H)$,
\begin{equation}\label{vf}
|\{f\in F:\; v=v_f\}|<0.9^k\binom{n-1}{k-1}.
\end{equation}
 Further,
we split the vertex set $V(H)$ randomly into two parts, $U_1$ and $U_2$, where each vertex belongs to $U_1$ independently with probability $1/k$.
We call a set $f\in F$ {\sl proper} if $v_f\in U_1$ and $f\subseteq U_2$.

Let $X$  count the number of proper sets. Since
 $$
	P(f \text{ is proper})=\frac{1}{k}\cdot\left(\frac{k-1}{k}\right)^{k-1} \ge\frac{1}{k}\cdot\frac{1}{e}>\frac{1}{3k},
	$$
 by Claim~\ref{cl3},
$$
EX=\sum_{f\in F} P(f \text{ is proper})=|F|\cdot P(f \text{ is proper})\ge \frac{0.96^k}{12k}\binom {n-1}{k-1}.
$$
Thus, there exists a partition $(U_1,U_2)$ such that the number of proper sets $f$ is at least $\tfrac{0.96^k}{12k}\binom {n-1}{k-1}$.
For each $v\in U_1$, set
$$F_v=\{f\in F:\; v=v_f\mbox{ and }f\subset U_2\}\quad\mbox{and}\quad \phi_v=\frac{|F_v|}{\binom {n-1}{k-1}}.$$

By (\ref{vf}) and the above lower bound on the number of proper sets~$f$, we have $\sum_{v\in U_1}\phi_v>0.96^k/(12k)$ and, for each $v$, we have also $\phi_v<0.9^k$.
We partition the set $\{v\in U_1:\; F_v\neq\emptyset\}$ into three subsets $W_1, W_2, W_3$ so that the sums $S_i:=\sum_{v\in W_i}\phi_v$, $i=1,2,3$, are as close to each other as possible. This can be done, for instance, by a greedy algorithm which places  the vertices one after another into the set with the current minimum total of $\phi_v$'s. Then, assuming that $S_1\le S_2\le S_3$, we have
$$S_1>S_3-0.9^k\ge\frac13(S_1+S_2+S_3)-0.9^k\ge\frac14(S_1+S_2+S_3),$$
provided
$$0.9^k<\frac1{12}\times \frac1{12k}0.96^k=\frac1{144k}0.96^k\le \frac1{12}(S_1+S_2+S_3),$$
which is valid for $k\ge250$.
Hence, for each $i=1,2,3$,
$$\sum_{v\in W_i}|F_v|\ge \frac{0.96^k}{48k}\binom {n-1}{k-1}.$$

The sets $W_1,W_2,W_3$ generate a corresponding partition of the proper sets $f$ into `colors' $C_i=\bigcup_{v\in W_i}F_v.$
In order to complete the proof of Proposition \ref{relaxed}, it suffices to show that such a 3-coloring
contains a $(k-1)$-path of length three whose edges are colored with different colors. Such a path can be extended to a copy of $\ppk$ in $H$, yielding a contradiction.

Howeover, all sets $W_i$ are so dense that the existence of such a path is
 an easy consequence of Fact~\ref{f1}. Indeed, recall that in each color there are at least $0.96^k/(48k)\binom {n-1}{k-1}$ edges. Therefore, by Fact~\ref{f1},
in each color $C_i$, $i=1,2,3$, viewed as a $(k-1)$-graph, one can find a sub-hypergraph $G_i$ with
$$\delta(G_i)\ge\frac{0.96^k}{48k^2}\binom {n-2}{k-2}.$$ Moreover, $|V(G_i)|\ge 0.9n$, since otherwise for each vertex $v\in V(G_i)$,
$$
\deg_{G_i}(v)\le {0.9n\choose k-2}<\frac{0.9^{k-2}n^{k-2}}{(k-2)!}<\frac{0.96^k}{48k^2}\binom{n-2}{k-2}\le \delta_{G_i},
$$
where the penultimate inequality holds for $k\ge250$.
Consequently, the intersection of the vertex sets of these three
graphs, $U:=V(G_1)\cap V(G_2)\cap V(G_3)$, has size $|U|\ge 0.7n$.

Fix a vertex $v\in U$. Since $\deg_{G_1}(v)\ge\tfrac{0.96^k}{48k^2}\binom {n-2}{k-2}$ and the number of edges of $G_1$ with $f\cap U=\{v\}$ is at most $\binom{0.3n}{k-2}$, there exists an edge $f_1\in G_1$ and a vertex $w\in U$, $w\neq v$, such that $\{v,w\}\subset f_1\cap U$.  Moreover, since the number of edges of $G_2$ containing $v$ and another vertex of $f_1$ is $O(n^{k-3})$, we can find $f_2\in G_2$ such that $f_1\cap f_2=\{v\}$. Similarly, there exists  $f_3\in G_3$ such that
$f_3\cap (f_1\cup f_2)=\{w\}$. Then the  edges $f_2$, $f_1$, and $f_3$  form a desired copy of $P^{(k-1)}$ in~$F$. Finally, the edges $\{v_{f_i}\}\cup f_i$, $i=1,2,3$, create a $k$-path $\ppk$ in $H$, a contradiction.
\end{proof}

\bigskip

\begin{proof}[Proof of Lemma~\ref{l1}]

If $|H|<0.96^k\binom {n-1}{k-1}$ then the assertion obviously holds. Let us assume that $|H|\ge 0.96^k\binom {n-1}{k-1}$. Then, by Proposition \ref{relaxed}, there exists a vertex $v\in V(H)$ with
	$$
	\deg_H(v)\ge 0.9^k{n-1 \choose k-1}.
	$$
	Therefore, by Proposition \ref{f2} with $b=0.9^k$,
	$$
	\deg_H(v)\ge|H|-\left(1-\left(\frac{0.9^k}{k-1}\right)^\frac{1}{k-2}\right)^{k-1}{n-1\choose k-1}.
	$$
Thus, all we need to verify is that
	$$\left(1-\left(\frac{0.9^k}{k-1}\right)^\frac{1}{k-2}\right)^{k-1}<0.96^k.$$
	
	To this end, observe that
$$0.96^{k/(k-1)}>0.96^2>0.9,$$
while
$$1-\left(\frac{0.9^k}{k-1}\right)^\frac{1}{k-2}<0.9$$
is equivalent to
$$0.1^{k-2}(k-1)<0.9^k$$
which holds for $k\ge3$.
	\end{proof}

\bibliographystyle{plain}

\begin{thebibliography}{99}
   	
%\bibitem{EKR}
%P.~Erd\"os, C.~Ko, R.~Rado, {\sl Intersection theorems for systems of finite sets}, \emph{Quart. J. Math. Oxford Ser. (2)}, 12 (1961), 313--320.
\bibitem{FJS} Z.~F\"uredi, T.~Jiang, R.~Seiver, {\sl  Exact solution of the hypergraph Tur\'an
problem for $k$-uniform linear paths}, \emph{Combinatorica} 34(3) (2014)  299--322

\bibitem{GR}
A.~Gy\'{a}rf\'{a}s, G.~Raeisi, {\sl  The Ramsey number of loose triangles and quadrangles in hypergraphs}, \emph{Electron. J. Combin.}, 19(2) (2012), \# R30.
%\bibitem{HK} J.~Han, Y.~Kohayakawa, { \sl Maximum size of a non-trivial intersecting uniform family which is not a
%subfamily of the Hilton-Milner family}, \emph{Proc. Amer. Math. Soc.}, 145(1) (2017), 73--87.
%\bibitem{HM}
%A.J.W.~Hilton, E.C.~Milner, {\sl Some intersection theorems for systems of finite sets}, \emph{Quart. J. Math. Oxford Ser. (2)}, 18 (1967), 369--384.
\bibitem{J}
E.~Jackowska, {\sl  The 3-color Ramsey number for a 3-uniform loose path of length 3}, \emph{Australasian J. Combin}, 63(2)
(2015), 314--320.
%\bibitem{JPRt}
%E.~Jackowska, J.~Polcyn, A.~Ruci\'nski, {\sl Tur\'an numbers for 3-uniform linear paths of length 3}, \emph{Electron. J. Combin.}, 23(2) (2016), \#P2.30
\bibitem{JPR} E.~Jackowska, J.~Polcyn, A.~Ruci\'nski,
{\sl  Multicolor Ramsey numbers and restricted Tur\'an
	numbers for the loose 3-uniform path of length three}, {\tt arXiv:1506.03759v1}, submitted.
\bibitem{KMW}
P. Keevash, D. Mubayi, R. M. Wilson, {\sl Set systems with no singleton intersection}, \emph{SIAM J. Discrete Math.}
20 (2006), 1031-1041.
%\bibitem{KM}
 %A.~Kostochka, D.~Mubayi, {\sl The structure of large intersecting families},
 %\emph{Proc. Amer. Math. Soc.}, {\tt  arXiv:1602.01391}, to appear.
 \bibitem{LP} T.~{\L}uczak, J.~Polcyn, {\sl  On the multicolor Ramsey number for 3-paths of length three},   \emph{Electron. J. Combin.}, 24(1) (2017), \#P1.27.
\bibitem{Omidi}
G.R.~Omidi, M.~Shahsiah, {\sl  Ramsey numbers of 3-uniform loose paths and loose cycles},
\emph{J. Comb. Theory, Ser. A}, 121 (2014), 64--73.
\bibitem{P}
J.~Polcyn, {\sl  One more Tur\'an number and  Ramsey number for the loose 3-uniform path of length three}, \emph{Discuss. Math. Graph Theory}, 37 (2017) 443--464.
%\bibitem{PRIF}
%J.~Polcyn, A.~Ruci\'nski, {\sl  A hierarchy of maximal intersecting triple systems}, {\tt %arXiv:1608.06114}, submitted.
\bibitem{PR}
J.~Polcyn, A.~Ruci\'nski, {\sl Refined Tur\'an numbers and Ramsey numbers for the loose 3-uniform path of length three}, \emph{Discrete Mathematics}, 340 (2017) 107--118.
\bibitem{Radz} S.P.~Radziszowski, {\sl Small Ramsey numbers},
\emph{Electron. J. Combin.}, Dynamic Survey, DS\#1

\end{thebibliography}

\end{document}